\documentclass[12pt]{amsart}

\usepackage{mathrsfs,amssymb}

\newtheorem{theorem}{Theorem}[section]
\newtheorem{lemma}[theorem]{Lemma}

\theoremstyle{definition}

\theoremstyle{remark}

\numberwithin{equation}{section}

\let \e=\varepsilon

\let \a=\alpha
\let \f=\varphi
\let \b=\beta

\let \O=\Omega

\let \ga=\gamma

\begin{document}

\title[The John-Nirenberg inequality]
{The John-Nirenberg inequality with sharp constants}

\author{Andrei K. Lerner}
\address{Department of Mathematics,
Bar-Ilan University, 5290002 Ramat Gan, Israel}
\email{aklerner@netvision.net.il}

\begin{abstract} 
We consider the one-dimensional John-Nirenberg inequality: 
$$
|\{x\in I_0:|f(x)-f_{I_0}|>\a\}|\le C_1|I_0|\exp\Big(-\frac{C_2}{\|f\|_{*}}\a\Big).
$$
A. Korenovskii found that the sharp $C_2$ here is $C_2=2/e$. It is shown in this paper
that if $C_2=2/e$, then the best possible $C_1$ is $C_1=~\frac{1}{2}e^{4/e}$.   
\end{abstract}



\maketitle

\section{Introduction}
Let $I_0\subset {\mathbb R}$ be an interval and let $f\in L(I_0)$. Given a subinterval $I\subset I_0$, set 
$f_I=\frac{1}{|I|}\int_If$ and $\O(f;I)=\frac{1}{|I|}\int_I|f(x)-f_I|dx.$

We say that $f\in BMO(I_0)$ if $\displaystyle\|f\|_{*}\equiv\sup_{I\subset I_0}\O(f;I)<\infty$. The classical John-Nirenberg 
inequality \cite{JN} says that there are $C_1,C_2>0$ such that for any $f\in BMO(I_0)$, 
$$
|\{x\in I_0:|f(x)-f_{I_0}|>\a\}|\le C_1|I_0|\exp\Big(-\frac{C_2}{\|f\|_{*}}\a\Big)\quad(\a>0).
$$

A. Korenovskii \cite{k1} (see also \cite[p. 77]{k2}) found the best possible constant $C_2$ in this inequality,
namely, he showed that $C_2=2/e$:
\begin{equation}\label{jnr}
|\{x\in I_0:|f(x)-f_{I_0}|>\a\}|\le C_1|I_0|\exp\Big(-\frac{2/e}{\|f\|_{*}}\a\Big)\quad(\a>0),
\end{equation}
and in general the constant $2/e$ here cannot be increased. 

A question about the sharp $C_1$ in (\ref{jnr}) remained open. 
In \cite{k1}, (\ref{jnr}) was proved with $C_1=e^{1+2/e}=5.67323\dots$. The method of the proof in \cite{k1} was based on the Riesz sunrise lemma
and on the use of non-increasing rearrangements. In this paper we give a different proof of (\ref{jnr}) yielding the sharp constant $C_1=\frac{1}{2}e^{4/e}=2.17792\dots$. 
\begin{theorem}\label{mr} Inequality (\ref{jnr}) holds with $C_1=\frac{1}{2}e^{4/e}$, and this constant is best possible. 
\end{theorem} 

We also use as the main tool the Riesz sunrise lemma. But instead of the rearrangement inequalities, we obtain a direct pointwise estimate for any $BMO$-function (see
Theorem \ref{main} below). The proof of this result is inspired (and close in spirit) by a recent decomposition of an arbitrary measurable function in terms of mean oscillations (see \cite{H,L}). 

We mention several recent papers \cite{SV,VV} where sharp constants in some different John-Nirenberg type estimates were found by means of the Bellman function method.  

\section{Proof of Theorem \ref{mr}}
We shall use the following version of the Riesz sunrise lemma \cite{Klemes}. 

\begin{lemma}\label{riesz}
Let $g$ be an integrable function on some interval $I_0\subset
{\mathbb R}$, and suppose $g_{I_0}\le \a$.
Then there is at most countable family of pairwise disjoint
subintervals $I_j\subset I_0$ such that
$g_{I_j}=\a$, and
$g(x)\le \a$ for almost all $x\in I_0\setminus\left(\cup_{j}I_j \right)$.
\end{lemma}

Observe that the family $\{I_j\}$ in Lemma \ref{riesz} may be empty if $g(x)<\a$ a.e. on $I_0$. 

\begin{theorem}\label{main} Let $f\in BMO(I_0)$, and let $0<\ga<1$.
Then there is at most countable decreasing sequence of measurable sets $G_k\subset I_0$ such
that $|G_{k}|\le \min(2\ga^k,1)|I_0|$ and for a.e. $x\in I_0$,
\begin{equation}\label{ineq}
|f(x)-f_{I_0}|\le \frac{\|f\|_{*}}{2\ga}\sum_{k=0}^{\infty}\chi_{G_k}(x).
\end{equation}
\end{theorem}

\begin{proof} Given an interval $I\subseteq I_0$, set $E(I)=\{x\in I:f(x)>f_{I}\}$.
Let us show that there is at most countable family of pairwise disjoint
subintervals $I_j\subset I_0$ such that $\sum_j|I_j|\le \ga|I_0|$ and for a.e. $x\in I_0$,
\begin{equation}\label{claim}
(f-f_{I_0})\chi_{E(I_0)}\le \frac{\|f\|_{*}}{2\ga}\chi_{E(I_0)}+\sum_j(f-f_{I_j})\chi_{E(I_j)}.
\end{equation}

We apply Lemma \ref{riesz} with $g=f-f_{I_0}$ and $\a=\frac{\|f\|_{*}}{2\ga}$. One can assume that $\a>0$ and
the family of intervals $\{I_j\}$ from Lemma \ref{riesz} is non-empty (since otherwise (\ref{claim}) holds trivially
only with the first term on the right-hand side). Since $g_{I_j}=\a$, we obtain
\begin{eqnarray*}
\sum_j|I_j|=\frac{1}{\a}\int_{\cup_jI_j}(f-f_{I_0})dx&\le& \frac{1}{\a}\int_{\{x\in I_0:f(x)>f_{I_0}\}}(f-f_{I_0})dx\\
&=&\frac{1}{2\a}\O(f;I_0)|I_0|\le \ga|I_0|.
\end{eqnarray*}
Further, $f_{I_j}=f_{I_0}+\a$, and hence
$$
f-f_{I_0}=(f-f_{I_0})\chi_{I_0\setminus \cup_jI_j}+\a\chi_{\cup_jI_j}+\sum_j(f-f_{I_j})\chi_{I_j}.
$$
This proves (\ref{claim}) since $f-f_{I_0}\le \a$ a.e. on $I_0\setminus \cup_jI_j$.

The sum on the right-hand side of (\ref{claim}) consists of the terms of the same form as the left-hand side. Therefore, one can proceed iterating~(\ref{claim}). Denote $I_j^1=I_j$, and let $I_j^k$ be the intervals obtained
after the $k$-th step of the process. Iterating (\ref{claim}) $m$ times yields
$$
(f-f_{I_0})\chi_{E(I_0)}\le \frac{\|f\|_{*}}{2\ga}\sum_{k=0}^{m}\sum_j
\chi_{E(I_j^k)}(x)+\sum_i(f-f_{I_i^{m+1}})\chi_{E(I_i^{m+1})}
$$
(where $I_j^0=I_0$). If there is $m$ such that for any $i$ each term of the second sum is bounded trivially by $\frac{\|f\|_{*}}{2\ga}\chi_{E(I_i^{m+1})}$,
we stop the process, and we would obtain the finite sum with respect to $k$. Otherwise, let $m\to\infty$. Using that
$$|\cup_{i}I_i^{m+1}|\le \ga|\cup_lI_l^{m}|\le\dots\le \ga^{m+1}|I_0|,$$
we get that the support of the second term will tend to a null set. Hence, setting $E_k=\cup_{j}E(I_j^k)$, for a.e  $x\in E(I_0)$ we obtain
\begin{equation}\label{eq1}
(f-f_{I_0})\chi_{E(I_0)}\le \frac{\|f\|_{*}}{2\ga}\Big(\chi_{E(I_0)}(x)+\sum_{k=1}^{\infty}\chi_{E_k}(x)\Big).
\end{equation}
Observe that $E(I_j)=\{x\in I_j:f(x)>f_{I_0}+\a\}\subset E(I_0)$. From this and from the above process we easily get that
$E_{k+1}\subset E_k$. Also, $E_k\subset \cup_jI_j^k$, and hence $|E_k|\le \ga^k|I_0|$.

Setting now $F(I)=\{x\in I:f(x)\le f_I\}$, and applying the same argument to $(f_{I_0}-f)\chi_{F(I)},$ we obtain
\begin{equation}\label{eq2}
(f_{I_0}-f)\chi_{F(I_0)}\le \frac{\|f\|_{*}}{2\ga}\Big(\chi_{F(I_0)}(x)+\sum_{k=1}^{\infty}\chi_{F_k}(x)\Big),
\end{equation}
where $F_{k+1}\subset F_k$ and $|F_k|\le \ga^k|I_0|$. Also, $F_k\cap E_k=\emptyset$. Therefore, summing (\ref{eq1}) and (\ref{eq2}) and
setting $G_0=I_0$ and $G_k=E_k\cup F_k, k\ge 1$, we get (\ref{ineq}).
\end{proof}

\begin{proof}[Proof of Theorem \ref{mr}] Let us show first that the best possible $C_1$ in 
(\ref{jnr}) satisfies $C_1\ge\frac{1}{2}e^{4/e}$. It suffices to give an example of $f$ on $I_0$ such that for any $\e>0$,
\begin{equation}\label{exa}
|\{x\in I_0:|f(x)-f_{I_0}|>2(1-\e)\|f\|_{*}\}|=|I_0|/2.
\end{equation}

Let $I_0=[0,1]$ and take $f=\chi_{[0,1/4]}-\chi_{[3/4,1]}$. Then $f_{I_0}=0$. Hence, (\ref{exa}) would follow from $\|f\|_{*}=1/2$. 
To show the latter fact, take an arbitrary $I\subset I_0$. It is easy to see that computations reduce to the following cases: $I$ contains only 
$1/4$ and $I$ contains both $1/4$ and $3/4$. 

Assume that $I=(a,b), 1/4\in I,$ and $b<3/4$. Let $\a=\frac{1}{4}-a$ and $\b=b-\frac{1}{4}$. Then
$f_I=\a/(\a+\b)$ and 
$$\O(f;I)=\frac{2}{\a+\b}\int_{\{x\in I:f>f_I\}}(f-f_I)=\frac{2\a\b}{(\a+\b)^2}\le 1/2$$
with $\O(f;I)=1/2$ if $\a=\b$. 

Consider the second case. Let $I=(a,b), a<1/4$ and $b>3/4$. Let $\a$ be as above and $\b=b-\frac{3}{4}$. Then 
$$\O(f;I)=\frac{2}{\a+\b+1/2}\int_{\{x\in I:f>f_I\}}(f-f_I)=\frac{4\a(4\b+1)}{(2\a+2\b+1)^2}.$$
Since 
$$\sup_{0\le\a,\b\le1/4}\frac{4\a(4\b+1)}{(2\a+2\b+1)^2}=1/2,$$
this proves that $\|f\|_{*}=1/2.$ Therefore, $C_1\ge\frac{1}{2}e^{4/e}$. Let us show now the converse inequality. 

Let $f\in BMO(I_0)$. Setting $\psi(x)=\sum_{k=0}^{\infty}\chi_{G_k}(x),$ where $G_k$ are from Theorem \ref{main}, we have
\begin{eqnarray*}
|\{x\in I_0:\psi(x)>\a\}|&=&\sum_{k=0}^{\infty}|G_k|\chi_{[k,k+1)}(\a)\\
&\le& |I_0|\sum_{k=0}^{\infty}\min(1,2\ga^k)\chi_{[k,k+1)}(\a).
\end{eqnarray*}
Hence, by (\ref{ineq}),
\begin{eqnarray*}
|\{x\in I_0:|f(x)-f_{I_0}|>\a\}|&\le& |\{x\in I_0:\psi(x)>2\ga\a/\|f\|_{*}\}|\\
&\le& |I_0|\sum_{k=0}^{\infty}\min(2\ga^k,1)\chi_{[k,k+1)}(2\ga\a/\|f\|_{*}).
\end{eqnarray*}
This estimate holds for any $0<\ga<1$. Therefore, taking here the infimum over $0<\ga<1$, we obtain
$$
|\{x\in I_0:|f(x)-f_{I_0}|>\a\}|\le \f\Big(\frac{2/e}{\|f\|_{*}}\a\Big)|I_0|,
$$
where
$$\f(\xi)=\inf_{0<\ga<1}\sum_{k=0}^{\infty}\min(2\ga^k,1)\chi_{[k,k+1)}(\ga e\xi).$$
Thus, the theorem would follow from the following estimate:
\begin{equation}\label{mp}
\f(\xi)\le \frac{1}{2}e^{\frac{4}{e}-\xi}\quad (\xi>0).
\end{equation}

It is easy to see that $\f(\xi)=1$ for $0<\xi\le 2/e$, and in this case (\ref{mp}) holds trivially. 
Next, $\f(\xi)=\frac{2}{e\xi}$ for $2/e\le\xi\le 4/e$. Using that the function $e^{\xi}/\xi$ is increasing on $(1,\infty)$
and decreasing on $(0,1)$, we get $\displaystyle\max_{\xi\in[2/e,4/e]}2e^{\xi}/e\xi=\frac{1}{2}e^{4/e},$ verifying (\ref{mp}) for 
$2/e\le\xi\le 4/e$. 
 
For $\xi\ge 1$ we estimate $\f(\xi)$ as follows. Let $\xi\in [m,m+1),m\in {\mathbb N}$.
Taking $\ga_i=i/e\xi$ for $i=m$ and $i=m+1$, we get
\begin{eqnarray}
\f(\xi)&\le& 2\min\Big(\Big(\frac{m}{e\xi}\Big)^m,\Big(\frac{m+1}{e\xi}\Big)^{m+1}\Big)\nonumber\\
&=&2\Big(\Big(\frac{m}{e\xi}\Big)^m\chi_{[m,\xi_m]}(\xi)+\Big(\frac{m+1}{e\xi}\Big)^{m+1}\chi_{[\xi_m,m+1)}(\xi)\Big),\label{inm}
\end{eqnarray}
where
$\xi_m=\frac{1}{e}\frac{(m+1)^{m+1}}{m^m}.$ Using that the function $e^{\xi}/\xi^m$ is increasing on $(m,\infty)$
and decreasing on $(0,m)$, by (\ref{inm}) we obtain that for $\xi\in [m,m+1)$,
\begin{eqnarray*}
\f(\xi)e^{\xi}&\le& 2\Big(\frac{m}{e\xi_m}\Big)^me^{\xi_m}=2\left(\frac{e^{\frac{1}{e}(1+1/m)^m}}{(1+1/m)^m}\right)^{m+1}\equiv c_m. 
\end{eqnarray*}

Let us show now that the sequence $\{c_m\}$ is decreasing. This would finish the proof since $c_1=\frac{1}{2}e^{4/e}$. 
Let $\eta(x)=(1+1/x)^x$ and $\nu(x)=(e^{\eta(x)/e}/\eta(x))^{x+1}.$
Then $c_m=2\nu(m)$ and hence it suffices to show that $\nu'(x)<0$ for $x\ge 1$. We have
$$\nu'(x)=\nu(x)\Big(\log\frac{e}{\eta(x)}-(1-\eta(x)/e)\log(1+1/x)^{1+x}\Big).$$

Since $\eta(x)(1+1/x)>e$, we get $\mu(x)=\frac{\eta(x)}{e-\eta(x)}>x$. From this and from the fact that the function 
$(1+1/x)^{1+x}$ is decreasing we obtain
$$
(e/\eta(x))^{\frac{1}{1-\eta(x)/e}}=(1+1/\mu(x))^{1+\mu(x)}<(1+1/x)^{1+x},
$$
which is equivalent to that $\nu'(x)<0$. 
\end{proof}

\end{document}